\documentclass[12pt]{extarticle}
\usepackage[utf8]{inputenc}
\usepackage{standalone}
\usepackage{amsmath}
\usepackage{amssymb}
\usepackage{amsbsy}
\usepackage{amsfonts}
\usepackage{titlesec}
\usepackage{caption}
\usepackage{tikz-cd}
\usepackage{amsthm}
\usepackage{geometry,graphicx,color}
\usepackage{mathtext}
\usepackage{hyperref}
\usepackage{enumitem}
\usepackage{chngcntr}
\theoremstyle{plain}
\newtheorem{theorem}{Theorem}
\newtheorem{lemma}[theorem]{Lemma}
\newtheorem{prop}[theorem]{Proposition}

\newcommand{\xar}[1]{\ensuremath{\xrightarrow{#1}}}
\newcommand{\mf}[1]{\pmb #1 }
\newcommand{\mc}[1]{\mathcal{#1}}
\newcommand{\mr}[1]{\mathrm{#1}}
\newcommand{\mb}[1]{\mathbf{#1}}
\newcommand{\R}{\mathbb{R}}

\newcommand{\Z}{\mathbb{Z}}

\newcommand{\U}{\mathrm{U(1)}}

\newcommand{\on}[1]{\operatorname{#1}}
\newcommand{\la}{\langle}
\newcommand{\ra}{\rangle}	
\newcommand{\ol}[1]{\overline{#1}}
\newcommand{\ul}[1]{\underline{#1}}
\newcommand{\op}{\mathrm{op}}
\newcommand{\N}{{\pmb  \triangle}}
\newcommand{\NN}{{\underline{\pmb \triangle}}}
\newcommand{\ca}{\circlearrowright}
\renewcommand{\c}[1]{({#1})_\circlearrowright}

\newcommand{\p}[1]{\subsection{#1}}
\titleformat{\subsection}[runin]{\small\bfseries}{\textbf{\arabic{section}.\arabic{subsection}.}}{0.5em}{}[]

\titleformat{\section}[block]{\normalsize\bfseries}{\textbf{\arabic{section}.}}{0.5em}{}[]

\title{Minimal triangulations of circle bundles, circular permutations  and binary Chern cocycle}
\author{Nikolai Mn\" ev\thanks{PDMI RAS;  Chebyshev Laboratory, SPbSU. Research is supported by the Russian Science Foundation grant 19-71-30002.} \\  \href{mailto:mnev@pdmi.ras.ru}{mnev@pdmi.ras.ru} }

\begin{document}
\maketitle

\abstract{
We investigate a  PL topology question: which circle bundles can be triangulated
over a given triangulation of the base? The question gets a simple  answer emphasizing the role of minimal triangulations encoded by
local systems of circular permutations of vertices of the base simplices. The answer is based on an experimental fact: classical Huntington transitivity axiom for cyclic orders can be expressed as the universal binary Chern cocycle. }\\

\section{Introduction}
\p{}
Let $B$ a PL polyhedron. There is Weyl-Kostant correspondence between its integer cohomology classes $H^2(B;\Z)$ and isomorphism classes of circle bundles $S^1 \xar{}E \xar{p} B$ over $B$. The class of a bundle $p$ is its first Chern class $c_1(p) \in H^2(B;\Z)$. The one-to-one correspondence is provided by the  isomorphism
 \begin{equation}\label{WK}
 H^1(B;\ul S^1)\approx \check H^2(B;\Z)
 \end{equation}
 where   $H^1(B;\ul{S^1})$ is the first  sheaf cohomology group of $B$ with coefficients in the sheaf $\ul{S^1}$ of germs of $S^1$-valued functions on $B$,  $\check H^2(B;\Z)$ is the second \v Cech cohomology group of $B$ (\cite{Chern1977}, \cite[2.1]{Brylinski2008}).  A circle bundle over $B$ can be triangulated, i.e. there is a map $\mf E \xar{\mf p} \mf B$ of simplicial complexes and a pair of homeomorphisms $b,h$ making the following square commute
\begin{equation}
\begin{tikzcd}[column sep=small]
{|\mf E|} \arrow[d,"|\mf p|"] \arrow[r,"h"]& E 	\arrow[d,"p"] \\
{|\mf  B|}\arrow[r,"b"]                   & { B}
\end{tikzcd}
\end{equation}
We address the following question: \textit{if the base  triangulating complex $\mf B$ is fixed, then which circle bundles $p$ have triangulation over $\mf B$?} The answer is complete and a bit strange sounding in the case when $\mf B$, $\mf E$ are semi-simplicial sets and $\mf p$ is a singular map of semi-simplicial sets. A singular map of finite semi-simplicial sets is a natural generalization of a map of simplicial complexes to a more flexible combinatorial category which still functorially represents PL maps by geometric realization.  A semi-simplicial set has its simplices ordered. The orders create special   orientations of simplices and thus simplicial chain and cochain complexes $C_\bullet ^\triangle (\mf B; \Z)$, $C^\bullet _\triangle (\mf B; \Z)$ computing integer singular homology and cohomology of $B$.
The answer is as follows:
\begin{theorem} \label{thm01}
A circle bundle $p$ can be semi-simplicially triangulated over the base finite semi-simplicial set
$\mf B$ iff its integer Chern class $c_1(p)\in H^2(|\pmb B|; \Z)$ can be represented by a binary simplicial cocycle in $Z^2_\triangle (\mf B; \{0,1\}\subset\Z)$ having values 0 and 1 on 2-simplices. For classical  simplicial triangulations the condition  is necessary but not sufficient.
\end{theorem}
In particular we got Effortless \& Local construction of triangulated circle bundles over a triangulation of a closed surface.  In this situation any binary 2-cochain is cocycle. When the surface is oriented, the circle bundles are classified by those Chern numbers and we have a theorem:
\begin{theorem} \label{thm02}
	Let $\mf T$ triangulate an oriented closed surface. Then we can triangulate semi-simplicially over $\mf T$ any circle bundle with Chern number $c$ such that $$|{\boldmath c}| \leq \frac{1}{2}\# \mf T_2$$ When the equality holds, the triangulation can be only semi-simplicilal, not simplicial.
\end{theorem}

The Theorem \ref{thm01} sounds like a certain discrete relative of another  Weil-Kostant theorem -- the theorem on the ``prequantum bundle" \cite[2.2]{Brylinski2008}, saying that to a simplectic form $\omega \in \Omega^2(M)$ having integer periods on a differential manifold $M$ corresponds a circle bundle on $M$ with connection form which curvature is $\omega$. Here the role of the simplectic form is played by the binary simplicial cocycle,
the role of a connection is played by a certain ``minimal" triangulation which can be associated to any triangulation up to choices using our ``spindle contraction trick".  Such a minimal triangulation has the associated  Kontsevich piecewise-differential connection form, providing the rational local formulas of \cite{MS}.  Its curvature symplectic  form integrated over the base simplices and shifted by the standard 2-coboundary $\frac{1}{2}$ is exactly the integer binary cocycle.

\p{}
 The Theorem \ref{thm01} is based on  an observation, trick, formula and an experimental fact emphasizing the central role of circular permutations in the subject. We will  describe the plan of paper.

\bigskip
\noindent
First, we need to collect in Section \ref{ss} some stuff on semi-simplicial sets,  its geometric realizations and PL topology.

\bigskip
\noindent
Then we pass in Section \ref{neck} to the observation.  The observation was central in \cite{MS} -- the stalk of triangulation of an oriented circle bundle over an ordered $k$-simplex is identified with  an oriented necklace whose beads are labeled by vertices of the base simplex $0,1,...,k$. The beads correspond to the maximal $k+1$-di\-men\-sional simplices in the stalk.  In this correspondence,  stalks of the minimal triangulation go to circular permutations of vertices of the base.  A minimally triangulated circle bundle corresponds to a local system of circular  permutations of the base ordered simplices. These local systems are combinatorial sheaves on the base semi-simplicial complexes and they have a representing (or classifying) object -- the simplicial set $\pmb  SC$ of all circular permutations.

\bigskip
\noindent
In Section \ref{spdl} we discuss a trick of ``spindle contracion" in the triangulation of a circle bundle. The trick is a bundle ``simple map" from \cite{waldhausen2013}, and in our case it reduces a triangulation of a circle  bundle over a fixed simplicial base to a minimal triangulation over the same base.

\bigskip
\noindent
In Section \ref{01} we introduce the universal  binary  Chern cocycle formula for minimally triangulated circle bundles. It is a form of the local formula from \cite{MS}.

\bigskip
\noindent
In Section \ref{hunti} we relate Huntington cyclic order axioms and the local binary formula for the Chern class.  Condition of axiomatic extension of a cyclic order appears to be  exactly a binary form of Chern cocycle. A very small calculation unfolds the coincidence.

\bigskip
\noindent
In Section \ref{prt1} we assemble the proof of Theorem \ref{thm01}.
By spindle contraction and the formula we associate with any triangulation of a circle bundle over $\mf B$ a binary Chern 2-cocycle. This provides the ``if" direction of the statement.   By Huntington's axiomatic, using a binary 2-cocycle  we construct a unique minimally triangulated circle bundle having the cocycle as its Chern cocycle, completing the ``only if" direction.

\bigskip
\noindent
In Section \ref{prt2} the proof of Theorem \ref{thm02} is assembled.

\p{} It is clear that the subject fits into the topic of crossed simplicial groups and generalized orders (see for example \cite{Kapranov2014}), but we postpone this aspect for further investigations.

\section{Simplicial and semi-simplicial sets and complexes. \label{ss}}
\p{}
Semi-simplicial sets with  singular morphisms added were introduced in \cite{RSI} under the cryptic name ``ndc css" and show up in literature under random names. For example they are called ``trisps" in \cite{Kozlov2008}. Acknowledging the serious historical mess in the terminology we call them ``semi-simplicial complexes", due the same good and in some aspects better behavior as locally ordered classical simplicial complexes. One can imagine  the category of semi-simplicial complexes one as a subcategory of simplicial sets which  has the best possible behavior of its \textit{core} -- the set non-degenerate simplices relatively to maps. They have all finite limits and  useful colimits commuting with them.  The core of limits has an expression using Eilenberg-Zilber order product of simplices.  Kan's second normal subdivison functor $\on{Sd^2}$ acts functorially, producing classical simplicial  complexes with homeomorphic
geometric realization.  Therefore they have an  associated functorial PL structure on geometric realizations in finite case.  To summarize: singular morphisms of semi-simplicial complexes can be used to encode combinatorially PL maps of PL-polyhedra, for example PL fiber bundles.
The category of  semi-simplicial complexes has a natural Grothendieck topology generated by coverages   by non-degenerate simplices.
Generally all the cellular sheaf theory as in \cite{Curry2013} works similarly for semi-simplicial complexes. The site structure in the finite case is actually a generalization of P.S. Alexandroff non-Hausdorff topology on abstract classical simplicial complexes.

\p{} We denote $\N$ the category of finite linear orders $[k]=\{0,1,2,...,k\}$ and non-decreasing maps between them called operators. Injective maps are boundary operators, surjective - degeneracy operators.

\bigskip
Set-valued pre\-sheaves on $\N$  are simplicial sets. The category of simplicial sets denoted by $\hat{\N}$. For a simplicial set $\N^\op\xar{\pmb X} \on {\pmb{\mc S \mr{ets} }}$, elements of $\pmb X_k$ are called $k$-simplices. For a boundary operator $[m]\xar{\mu} [k]$ and a simplex $x \in \pmb X_k$, the  $m$-simplex $\mu^*(x)\in \pmb X_m$ is called the $\mu$-th boundary of $x$. The same for degeneracies.

\p{}

\bigskip
A part of category $\N$ generated by all injective maps denoted by $\NN$. Set-valued presheaves on $\NN$ are called semi-simplicial sets. They form the category $\hat \NN$.

\bigskip
 One can make from a simplicial set a semi-simplicial set by forgetting all the degeneracies. This provides a functor $\hat \N \xar{F} \hat{\NN}$ having left adjoint functor $S$.  The  theory of semi-simplicial complexes is based on the Rourke-Sanderson adjacency $S \dashv F$. Functor $S$ freely adds degeneracies to a semi-simplicial set making it a simplicial set. Completing the image of $S$ to a full subcategory in $\hat \N$ we obtain a full subcategory  $\tilde \NN$  of $\hat \N$ --  the category which we call the\textit{ category of semi-simplicial complexes}. 

 \bigskip
 We denote by $\la m\ra \xar{\la \mu \ra} \la k \ra$ the images of orders and operators under Yoneda embedding. We imagine them
as standard face and degenerecy maps of ordered abstract simplices. The Yoneda images of $\NN,\N$  belong to the $\tilde \NN$.

\bigskip
The category of singular morphisms $\on{Arr} \tilde \NN$ is the convenient category for triangulations of bundles by geometric realization.

\section{Triangulations and necklaces.} \label{neck}
Here we will repeat a few points from \cite{MS} in a way convenient for the current exposition.
\p{Triangulations of circle bundles.} Suppose we have a finite semi-simplicial complex $\pmb B$ and an  oriented  circle bundle $S^1 \xar{} E \xar{} |\pmb B | $  triangulated over $\pmb B$. I.e. a semi-simplical complex $\pmb B $ and a singular map $\pmb E \xar{\pmb p} \pmb B $ of semi-simplicial complexes for which there  exists a
homeomorphism $h$ making the diagram
commutative:
\begin{equation}
\begin{tikzcd}[column sep=small]
{|\mf E|} \arrow[dr,"|\mf p|"] \arrow[rr,"h"] & & E \arrow[dl,"p"] \\
& {|\mf  B|}                  &
\end{tikzcd}
\end{equation}
Homeomorphism  $h$  creates on $p$ the
structure of a PL oriented circle bundle. Any
two such  homeomorphisms create fiberwise PL
isomorphic structures. Moreover: over a PL polyhedral
base, the oriented $S^1$ bundles understood as
principle $U(1)$-bundles or as  oriented PL
fiber bundles are the same thing. On the total space $E$
one can always choose an interior flat  Euclidean metric making all
the fibers of $p$ of constant perimeter
($2\pi$ or $1$ or whatever makes   formulas
nicer). This will miraculously turn an oriented PL
$S^1$ bundle $p$ into $U(1)$ principal
bundle $p$ also in a unique up to $\U$-gauge
transformation way. Therefore if  $h$ exists,
than the  combinatorics of the map $\pmb p $
determines isomorphism the class of an $S^1$
bundle and hence its  Chern  class $c_1(p)\in H^2(B;\Z)$ in the
base,  by Weil-Kostant theorem.
\p{Simplicial circle bundle.}
Picking a base $k$-simplex  $\la k \ra \xar{x} \pmb B$ we can form a stalk of $\pmb p$ over  $x$ -- the pulback
$x^* \pmb E \xar{x^* \pmb p}\la k\ra $ -- an  \textit{elementary s.c. bundle} over a simplex. The bundle $p$ is oriented. The orientation  fixes a generator in the first integer simplicial homology group of the total complex $x^* \pmb E$. Simplicial boundary transition maps between the stalks of $\pmb p$ send generator to generator, representing the orientation as a constant local system on the base $\pmb B$ associated with the triangulation.
We call a \textit{simplicial circle bundle} (\textit{s.c. bundle}) on a semi-simplicial complex $\pmb B$ a local system of oriented elementary s.c. bundles   on $ \pmb B$ and orientation preserving transition boundary maps. It assembles by colimit in the category $\on{Arr} \tilde\NN$ of  singular morphisms to a map $\pmb E \xar{} \pmb B$ having the canonical  structure of a PL triangulated  $S^1$ bundle on the geometric realization (if $\mf B$ is finite) with a canonical structure of $\U$-principal bundle. (We are in a simple situation of a stack where elementary s.c. bundles and transition boundary maps are the ``descent data")

\p{Necklace of an elementary s.c. bundle.}

Now let $\pmb R \xar{\pmb e} \la k \ra$
be an elementary s.c. bundle over $\la k\ra$ having $n\geq k+1$ maximal $k+1$-dimensional simplices in the total complex $\pmb R$.
  The semi-simplicial bundle $\pmb e$ is determined by an oriented necklace $\mc N(\pmb e)$ whose $n$ beads are colored by vertices of the base simplex, i.e. the numbers $\{0,...,k\}$.
 Figure \ref{eb} presents a picture of an elementary s.c. bundle over the 1-simplex $\la 1\ra$.
 \begin{figure}[h!]
 	\captionsetup{font=small,width=0.8\textwidth}	
 	\begin{center}
 		\includegraphics[width=3.0in]{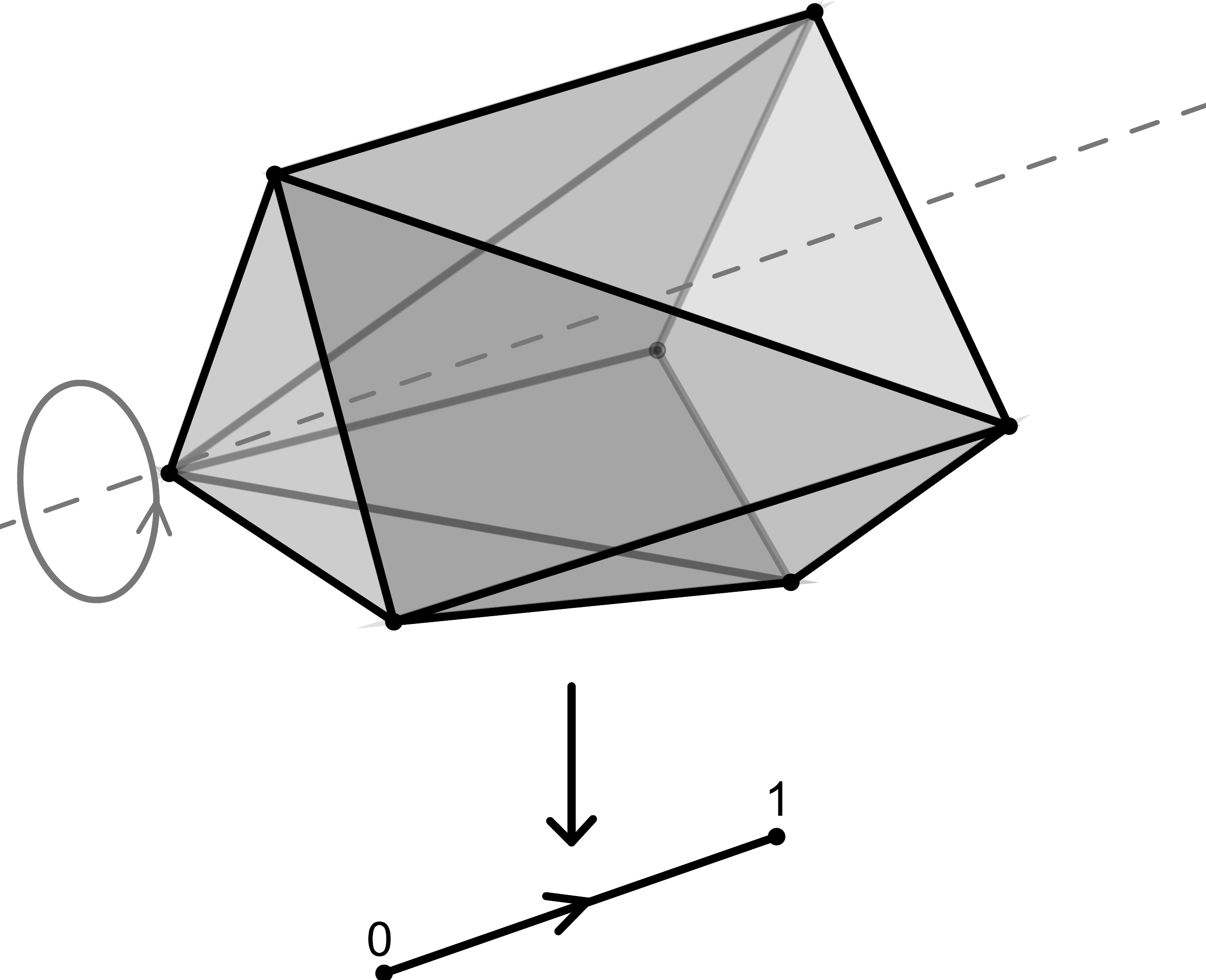} 
 		\caption{\small Elementary simplicial circle bundle \\ over the interval. \label{eb}}
 	\end{center}
 \end{figure}
 To an elementary	 simplicial circle bundle $\mf e$ over $\la k \ra$ having $n$ maximal $k+1$-dimensional simplices in the total space,   we associate a $k+1$-necklace $\mc N (\mf e)$, i.e. a length $n$ circular word  in the ordered alphabet of $k+1$ letters numbered by the vertices of the base simplex. Any $k+1$ - dimensional simplex of $\mf R$ has a single  edge which shrinks to a vertex $i$ of the base simplex by Yoneda  simplicial degeneration $\la k+1 \ra \xar{\la \sigma_i\ra} \la k\ra, i=0,1,2,...,k$.  Take the general fiber of the projection $|\mf e|$. It is a circle broken into $n$ intervals  oriented by the orientation of the bundle, and every interval on it is an intersection with a maximal $(k+1)$-simplex. The maximal simplex is uniquely named by a vertex of the base where its collapsing edge collapses.  This creates a coloring of the $n$ intervals by $k+1$ ordered vertices of base simplex.  Thus we got a necklace $\mc N (\mf e)$ out of the combinatorics of $\mf e$ (\cite[\S 16]{MS}).  The process is illustrated on Fig \ref{eb2}. The process is invertible: having an oriented  necklace $\vartheta$
 whose beads are colored by $[k]$ we can assemble an elementary oriented s.c. bundle $\pmb EC(\vartheta)\xar{\pmb ec(\vartheta)} \la k \ra$ as a colimit in $\tilde \NN \downarrow_{\la k \ra}$ (or $\hat \N \downarrow_{\la k \ra}$) of Yoneda degeneracies $\la k+1\ra\xar{\la \sigma_i\ra} \la k \ra$.

 \begin{figure}[h!]
 	\captionsetup{font=small,width=0.8\textwidth}	
 	\begin{center}
 		\includegraphics[width=5.0in]{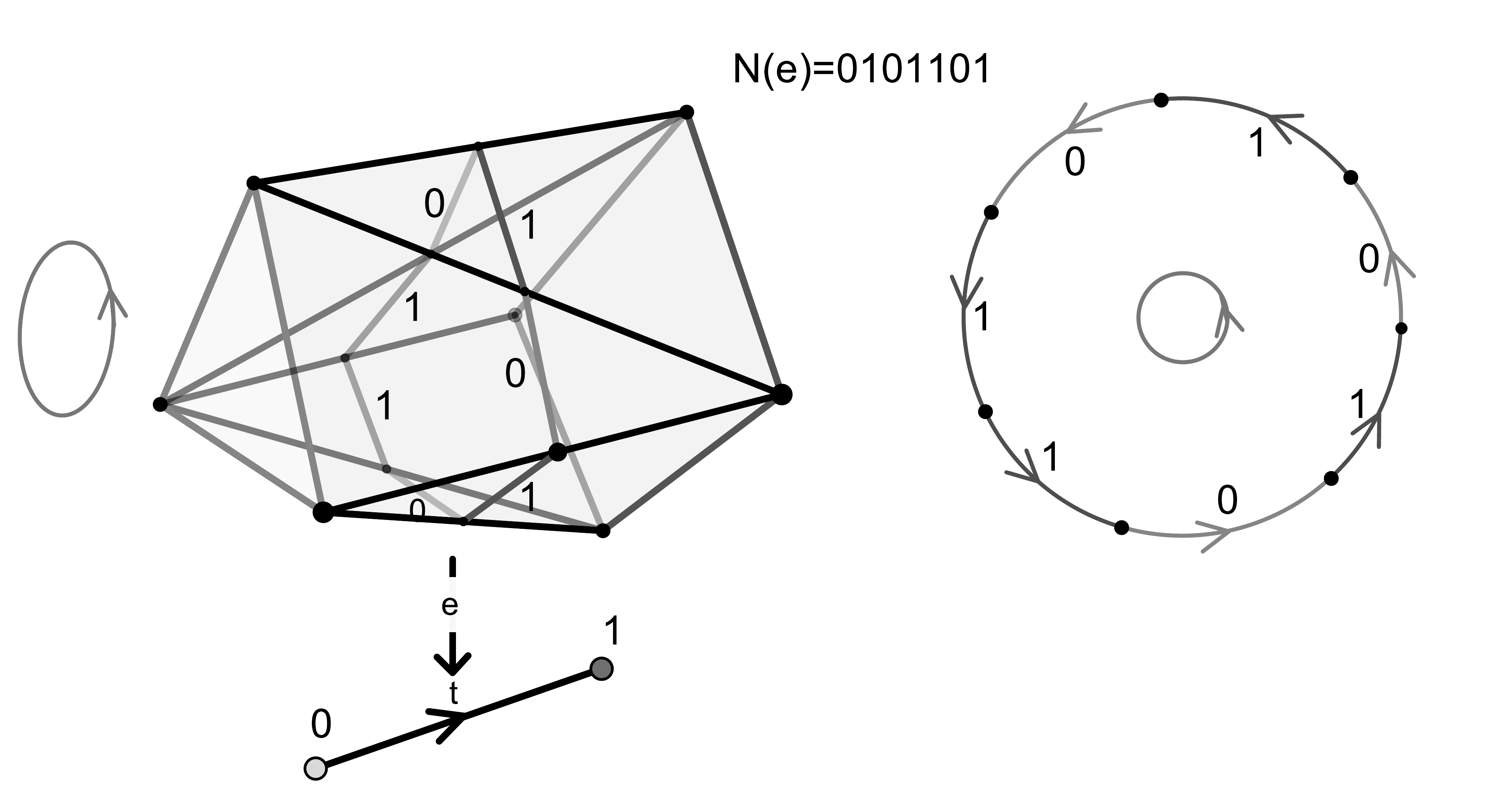}
 		\caption{ \label{eb2}}
 	\end{center}
 \end{figure}

\p{Local systems of oriented necklaces.} \label{neckloc}
Let   $\pmb E \xar{\pmb p} \pmb B $ be an s.c.  bundle, $\la k\ra \xar{x} \pmb B$ a simplex of base, $x^*\pmb E \xar{ x^*\pmb p } \la k \ra$ the corresponding subbundle and $\la k-1 \ra \xar{d_i x} \pmb B$ - $i$-th boundary of the simplex $x$. Then by construction the necklace $\mc N (d_i x)^* \pmb p$ is obtained from the necklace  $\mc N(x^* \pmb p)$ by deleting all the beads colored by $i$.
 But the face maps between  elementary subbundles contain more information, since the elementary bundles and the correspondent necklaces may have combinatorial automorphisms. Therefore they should be recorded in the descent data of the bundle. After this  fix,  the bundle $\pmb p$ is encoded in the  local system $\mc N(\pmb p)$ of  oriented necklaces   made of ordered vertices of the  base simplices and vise versa a local system $\Theta$ of necklaces on the base $\pmb B$  encode the bundle $\pmb EC(\Theta) \xar{\pmb ec (\Theta)} \pmb B$.

\p{Classical simplicial vs semi-simplicial triangulations.} Not every oriented necklace $\vartheta$ with beads colored by $[k]$ has as $\pmb EC(\vartheta)$ a classical simplicial complex. We say that $\vartheta$ has two colors $i,j \in [k]$  ``not mixed"  if, after deleting from $\vartheta$ all the  beads except those colored by $i,j$,  the remaining two sorts of  beads stay in two solid blocks (see proof of \cite[Lemma 0.1]{Mnev:2018}).
\begin{prop} \label{puresimp}
The complex  $\pmb EC(\vartheta)$ is classically simplicial iff $\vartheta$ has\\
1)  no less than 3 beads of each color and \\
2) any two pair of its colors are ``mixed".\\
\end{prop}
\begin{proof}
A semi-simplicial complex is classically simplicial iff its 1-dimen\-sional skeleton is classically simplicial, i.e is a graph without loops or multiple edges. The entrie 1-dimen\-sional skeleton of $\pmb E C (\vartheta)$ sits over the one-dimensional skeleton of the base simplex $\la k\ra$. The condition 1) guarantees that there are no loops or multiple edges in circles over vertices, the conditon 2) guarantees that there are no multiple edges in the total complex over the base edges (see proof of \cite[Lemma 0.1]{Mnev:2018}) ).
\end{proof}
If conditions of the Proposition \ref{puresimp} do not hold, then the  complex $\pmb EC(\vartheta)$ is essentially semi-simplicial, and some of its simplices can have glued vertices or two different 1-simplices have both vertices in common.
\p{Circular permutations and minimal elementary s.c. bundles.} \label{cm}
We arrived to our main objects.

\bigskip
\noindent
 If the oriented necklace $\vartheta $ has a single bead of every color from $[k]$ then $\vartheta$ is a \textit{circular permutation of $[k]$}. We denote by $\pmb SC$ the $\mathbb N$-graded set of circular permutations.

\bigskip
\noindent
 Denote  by $\pmb S$ the  $\mathbb N$-graded set with $\pmb S_k$ -- symmetric group of all permutations of $[k]$. A circular permutation in $\pmb SC_k$ is the same as the right coset of a permutation $\omega \in \pmb S_k$ by the right action of the cyclic subgroup $\pmb C_k$ of $\pmb S_k$. Thus we have a map of graded sets $\pmb S_k \xar{\ca} \pmb SC_k$ sending a permutation to its right cyclic coset, i.e. to a circular permutation.

\bigskip
\noindent
 Let us organize the correspondence $\ca$ in a way comparable with boundaries of necklaces from p. \ref{neck}.\ref{neckloc} and also add degeneracies, making $\ca$ a morpism of simplicial sets.

\bigskip
\noindent
 First add boundaries and degeneracies to the graded set $\pmb S$. Define  boundaries
$ \pmb S_k \xar{d_i} \pmb S_{k-1},i=0,...,k $ by deleting the element $\omega(i)$ from permutation $\omega$ followed by monotone reordering. Thus $\pmb S$ became a semi-simplicial set.

\bigskip
\noindent
Define degeneracies $\pmb S_k \xar{s_i} \pmb S_{k+1}, i =0,...,k$ by inserting  into a permutation a new element right after $\omega(i)$ on $\omega(i)+1$'st  place, with the value $\omega(i) + \frac{1}{2}$  and monotone reordering the values  to natural numbers.
Now  $\pmb S$ is a simplicial set.  This simplicial set of permutations $\pmb S$ is a classical object called the ``symmetric  crossed simplicial group", introduced independently  in \cite{FT1987}, \cite{Kras1987} and later in \cite{FL1991}.

\bigskip
\noindent
The map $\ca$ induces the similarly defined simplicial structure on $\pmb SC$ making the map $\ca$ simplicial.

\bigskip
\noindent
Now make the following definition
\begin{itemize}

\item For a circular permutation $\vartheta  \in \pmb SC_k $, we call \textit{minimal} the elementary s.c. bundle
$$\pmb E C(\vartheta) \xar{\pmb e c(\vartheta)} \la k \ra$$  We call minimal a s.c. bundle if all its stalks over simplices are minimal.
\end{itemize}
    Circular words and corresponding  bundles have no automorphisms. Therefore a bundle having minimal its stalks over all simplices is the same as a local system of circular permutations of the base simplices and its simplicial boundary maps. It is the same as a simplicial map  $\pmb B \xar{} \pmb SC$.

\bigskip
\noindent
We arrived to the point that the functor on $\tilde \NN$ assigning to a semi-simplicial set $\pmb B$ the set of all minimally triangulated circle bundles over $\pmb B$ is represented by $\pmb SC$.

\bigskip
\noindent
Actually, a minimal elementary s.c. bundle $\pmb ec(\vartheta) $ is the stalk of the simplicial map $\pmb S \xar{\ca} \pmb SC$ over the base simplex $\la k \ra \xar{\vartheta} \pmb SC$. 	Therefore $\ca$ is the universal minimal s.c. bundle over $\pmb SC$.  We don't prove  this fact in this paper.

\p{Geometry of minimal elementary s.c. bundles.}
Minimal elementary s.c. bundles are the same thing as the twisted product projection  $\pmb C \times_{\omega} \la k \ra \xar{} \la k \ra$ where $\pmb C$ is  Connes' cyclic crossed simplicial group or ``simplicial circle" $\pmb C$, $\omega \in \pmb  S_k$ is a permutation of the base vertices.

 Now we will  describe elementary minimal s.c. bundles
  geometrically  (see Figure \ref{ibund}, Figure \ref{021}).
Let $\omega \in \pmb S_k$ be a permutation and $\c{\omega} \in \pmb SC_k$ - the corresponding circular permutation. We construct the elementary s.c. bundle
 $$\pmb EC\c{\omega}\xar{\pmb ec\c{\omega}}\la k \ra$$
 by the following algorithm.
 Take the geometric prism $\Delta^k \times \Delta^1 \subset \R^k \times \R^1$, and number its verities by $[k]\times[1]$. Then apply the algorithm:
  on step 0 make a $k+1$-simplex which is
  convex hull of the bottom $k$-simplex with vertices $((0,0),...,(0,k))$ and the point
  $(\omega(0),1)$. The result  will have the top $k$-simplex with vertices
  $$(\omega(0),1),(\omega(1),0),...,(\omega(k),0)$$
  Then iterate, building the pile of $k+1$-simplices.
  On step $i$,
  add a $k+1$ simplex which is the convex hull of the point $(\omega(i),1)$ and the top $k$-simplex in the already   constructed pile. It is a very simple sort of ``shelling" process in simplicial combinatorics. On the step $k$ we will obtain a certain triangulation $\pmb E(\omega)$ of the prism $\Delta^k \times \Delta^1$. At the last step of the construction of $\pmb E C\c{\omega}$ -- the bundle corresponding to circular word $\c{\omega}$ -- we glue together the very top and the very bottom $k$-simplices.  It is possible to do  only semi-simplicially. The general  fiber of the projection intersects the interiors of $k+1$-simplices in the circular order $\c{\omega}$ and we could start from any cyclic shift of the word $\omega$ with the same result.
\begin{figure}[hbt]
	\captionsetup{font=small,width=0.8\textwidth}	
 \begin{center}
 \
  \includegraphics[width=1.5in]{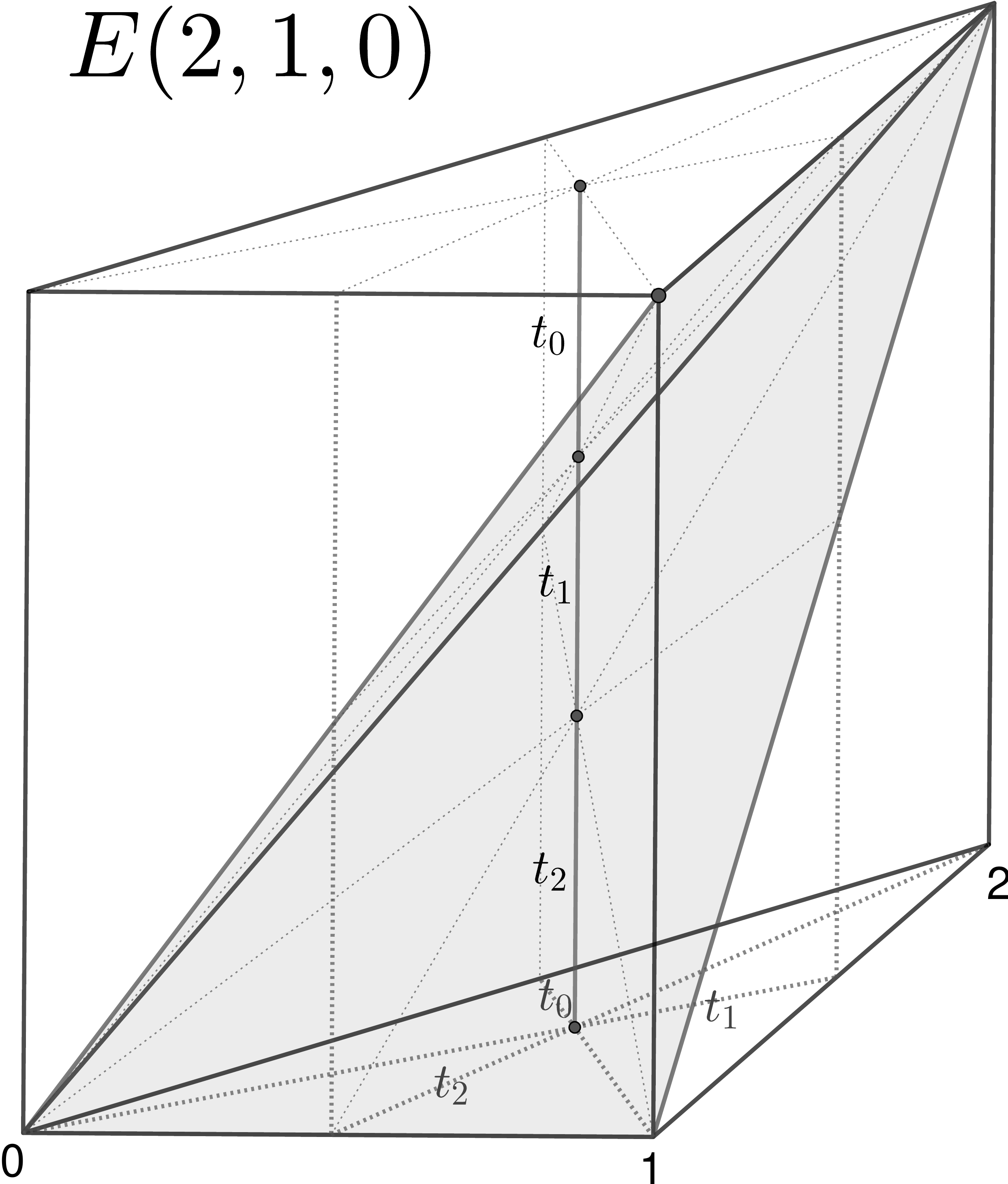}
  \includegraphics[width=1.5in]{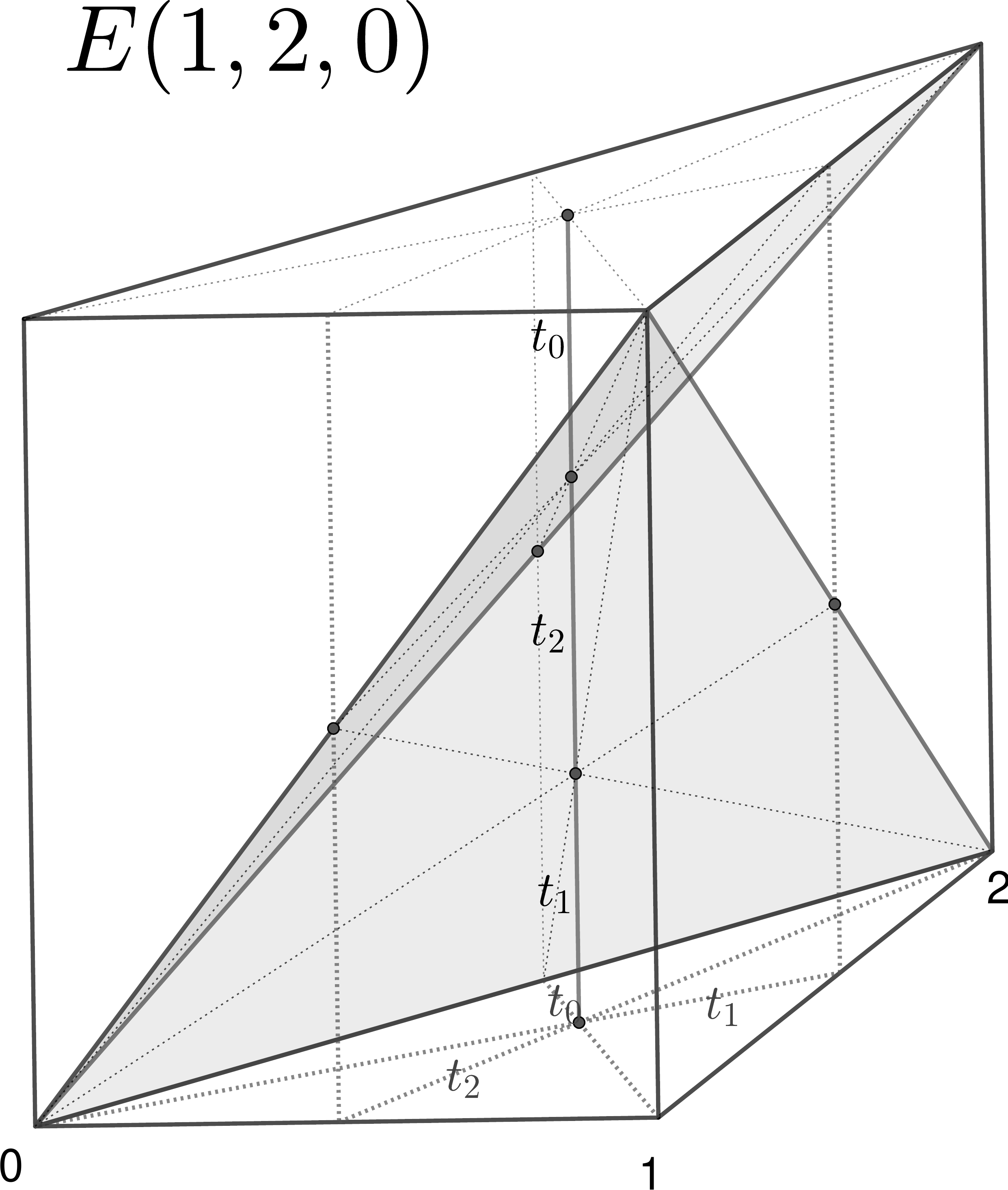}
  \caption{ \label{ibund}}\end{center}
\end{figure}
\begin{figure}[hbt] \label{021}
	\captionsetup{font=small,width=0.8\textwidth}	
	\begin{center}
		\includegraphics[width=1in]{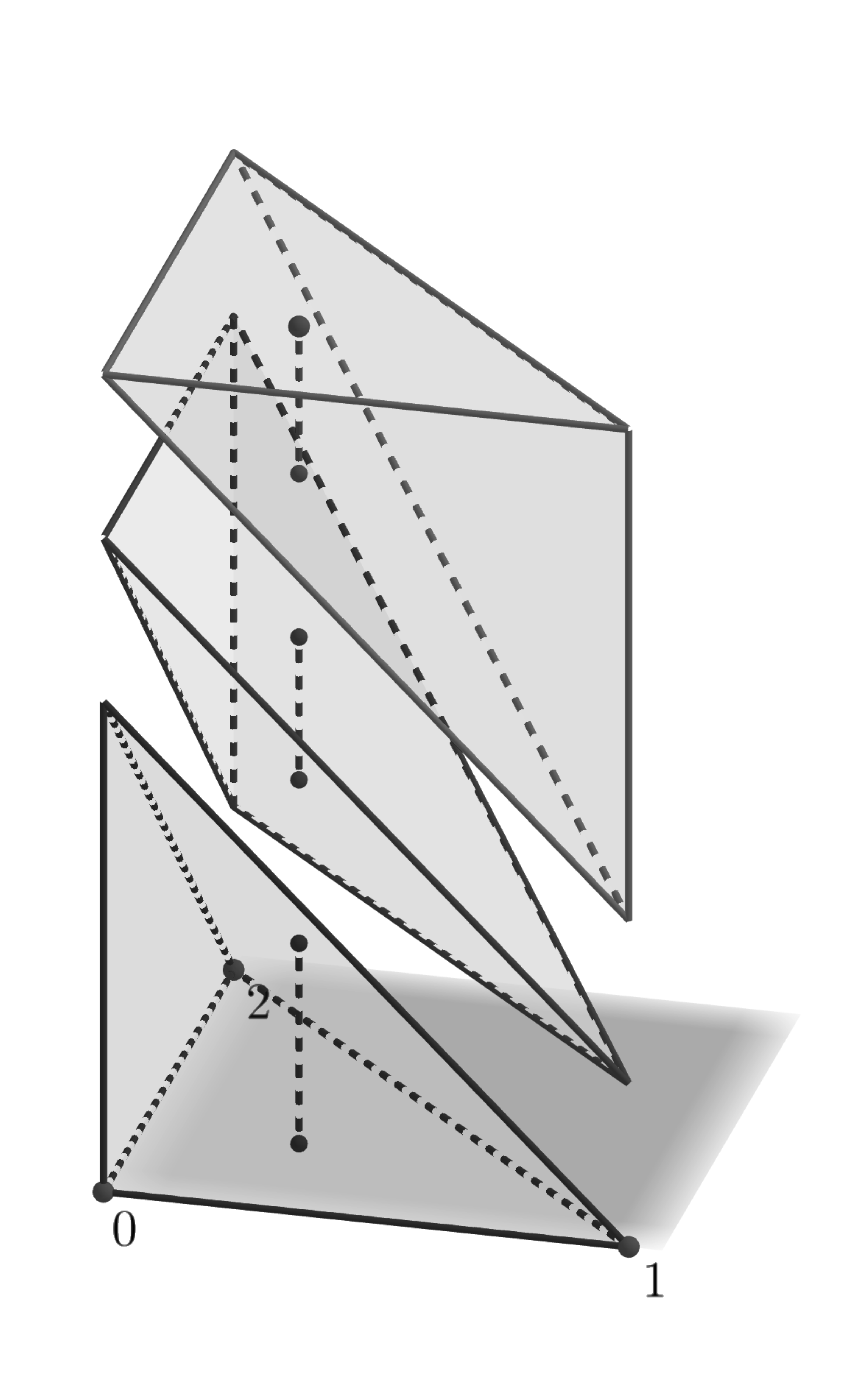}
		\includegraphics[width=1in]{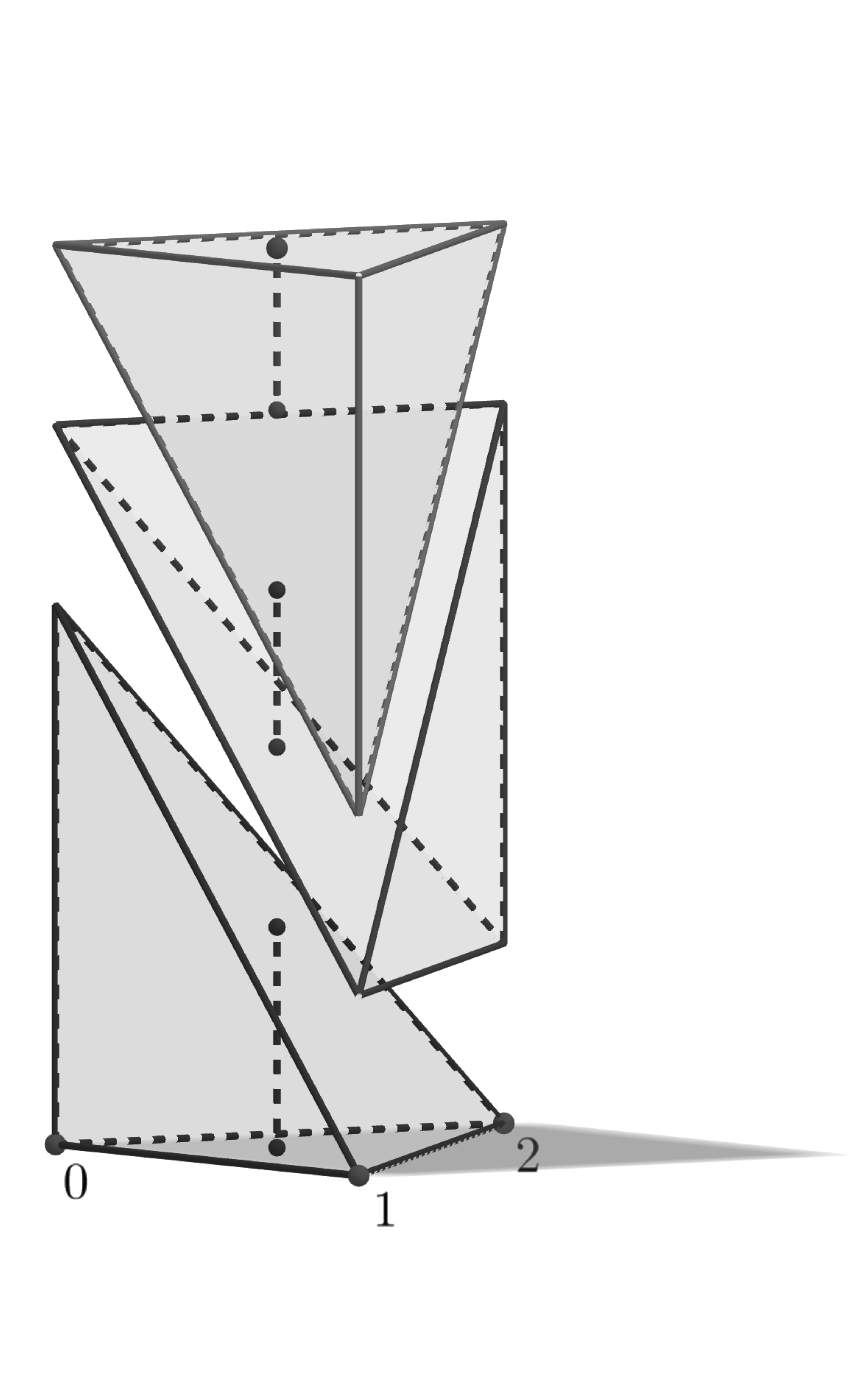}
		\includegraphics[width=1in]{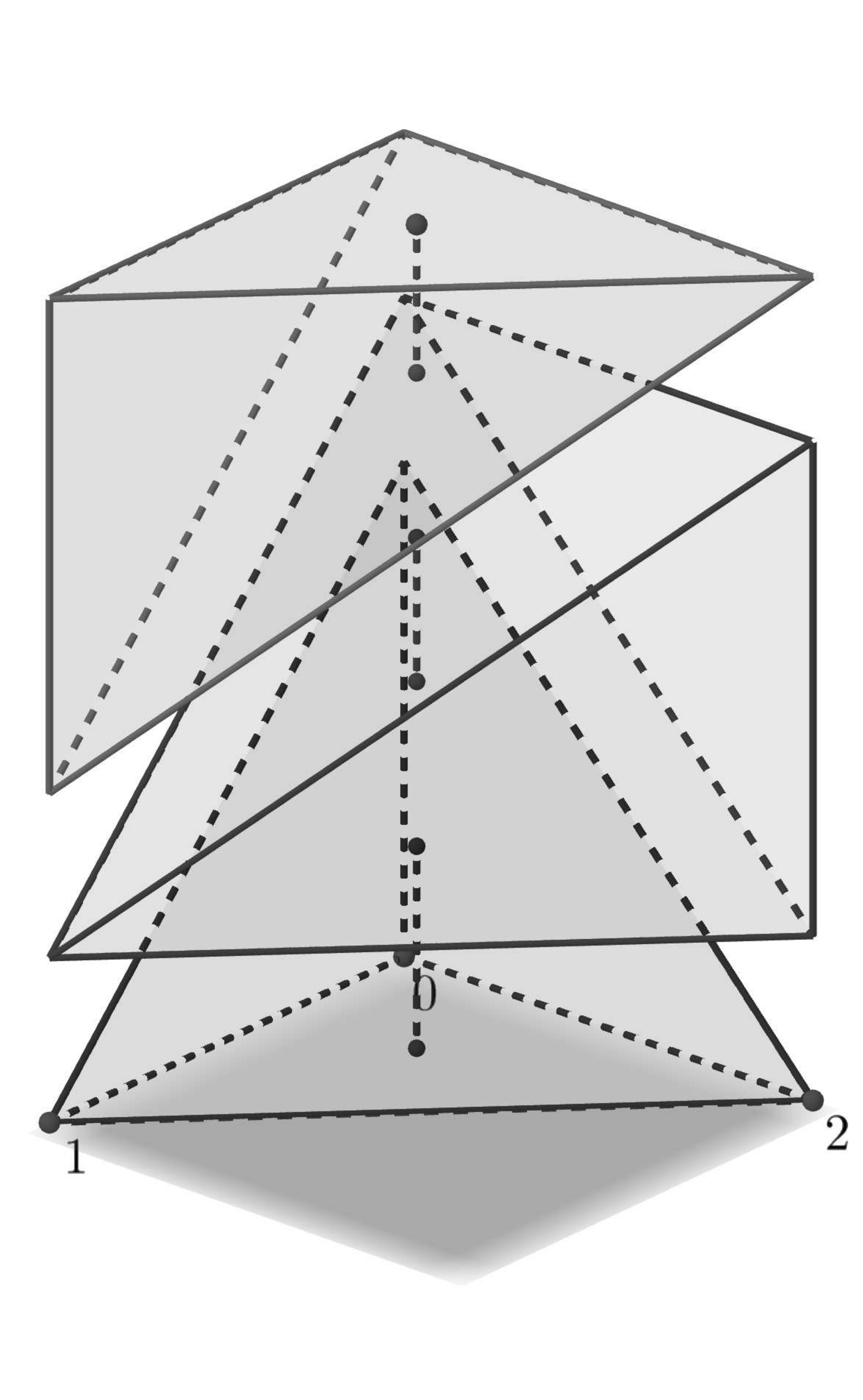}
		\caption{$E(0,2,1)$\label{021}}\end{center}
\end{figure}

\bigskip
\noindent
The most important for us are the circular permutations of $[0],[1],[2],[3]$  and corresponding minimal elementary s.c. bundles.
We have:
\begin{itemize}
\item single circular permutation of one element $\c{0}$;
\item  single circular permutation of two elements  $\c{0,1}$;
\item 2 circular permutations of three elements: even $\c{0,1,2}$ and odd $\c{2,1,0}$;
\item 6 circular permutations of  4 elements.
\end{itemize}
Those faces and boundaries  form  the  skeleton $\pmb SC(3)$ of $\pmb SC$ depicted as a hexagram on Figure \ref{hex}.  	
\begin{figure}[hbt]
	\captionsetup{font=small,width=0.8\textwidth}
	\centering
	\includegraphics[width=4in]{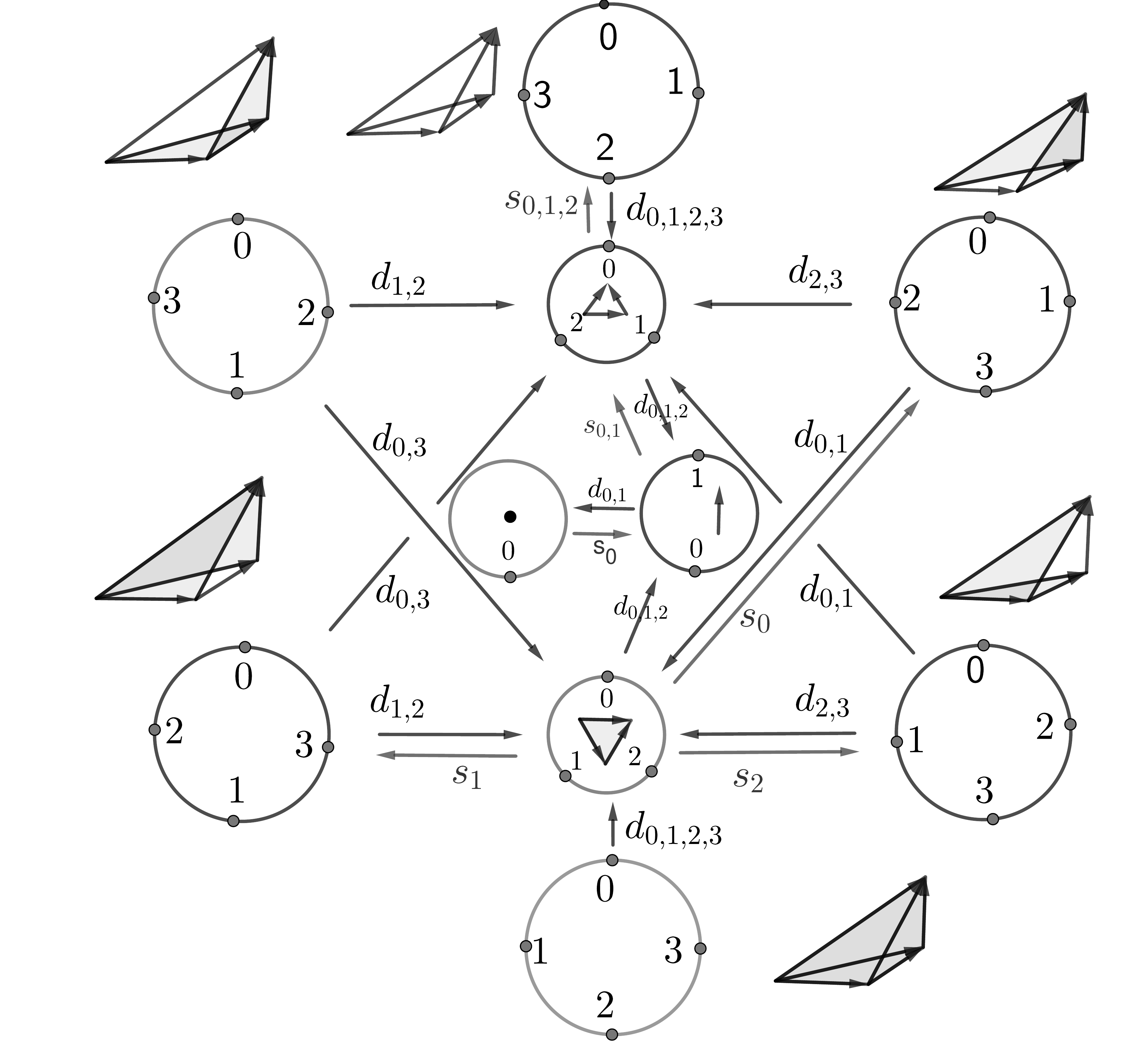}
	\caption{Skeleton $\mf S C (3)$ depicted as a hexagram. All the circles are oriented clockwise.  \label{hex}  }
\end{figure}
\section{Spindle contraction trick. } \label{spdl}
Here we will show how  up to a free choice,  one can reduce any s.c. bundle  to a minimal one preserving the bundle isomorphism class of the geometric realization.
In PL topology concordant fiber bundles are isomorphic and vice versa.
  We call two  s.c. bundles $\pmb p, \pmb q$ on $\pmb B$ strongly concordant
if there is a s.c. bundle $\pmb h$ on $\pmb B \times \la 1 \ra$ such that its restrictions on $\pmb B\times \la 0\ra$ and $\pmb B\times \la 1 \ra$ are $\pmb p$ and $\pmb q$.
Geometric realizations of concordant bundles are isomorphic.  We call by {\em circles of the s.c. bundle $\pmb p$} its 0-stalks over vertices $\pmb B_0$ of $\pmb B$.  The circle of $\pmb p$ is a semi-simplicial oriented circle consisting from vertices and oriented arcs contracting to a vertex of the base by the  bundle projection $\pmb p$.
\begin{prop}\label{conc}
Any  s.c. bundle is strongly  concordant to
a minimal s.c. bundle, uniquely  determined by a free choice  of a single arc in every circle of the bundle.
\end{prop}
We will prove the proposition after  introducing spindle contraction trick.
\p{}
Suppose that we have a vertex $v \in \pmb B_0$, circle $ c(v) $ over $v$ and an arc $a \in c_1(v)$. Consider the star of $a$ in $\pmb E$, $\on{st} a \xar{} \pmb E$ and the star of $v$, $\on{st}(v) \xar{} \pmb B$.
The projection $\pmb p$ induces a subbundle -- the ``spindle" ${\on{sp}}(a)$ undersood as a morpism in $\on{Arr} \tilde \NN$ :
\begin{equation} {\on{sp}}(a) :
\begin{tikzcd}
{\on{st}(a)} \arrow[d,"\pmb p'"'] \arrow[r]& \pmb E \arrow[d,"\pmb p"] \\
{\on{st}(v)}\arrow[r,""]                   & { \pmb B}
\end{tikzcd}
\end{equation}
Spindle's projection on the base can be understood as a morphism in $\on{Arr} \tilde \NN$
\begin{equation}
\begin{tikzcd}
 {\on{st}(v)}\arrow[d,"\on{Id}"']& {\on{st}(a)}\arrow[l,"\pmb p'"']  \arrow[d,"\pmb p'"]\\
 {\on{st}(v)}& \on{st}(v) \arrow[l,"\on{Id}"' ]
\end{tikzcd} : \ul{\on{sp}}(a)
\end{equation}
``Spindle contraction" $\pmb b/a$  is  colimit in $\on{Arr} \tilde \NN$ of the diagram
\begin{equation}
\begin{tikzcd}
{p'} \arrow[d,"\ul{\on{sp}}(a)"'] \arrow[r,"\on{sp}(a)"]& \pmb p \arrow[d,dashrightarrow,""] \\
{\on{Id}_{\on{st}(v)}}\arrow[r,dashrightarrow,""]                   & { \pmb p/a}
\end{tikzcd}
\end{equation}
Figures \ref{spind},\ref{spind2} illustrates spindle contractions. Figure \ref{spind2} illustrates commuting of 2-dimensional spindle contractions over 1-dimensional base, but  3-dimensional contractions over 2-dimesional base already don't commute.
\begin{figure}[!h]
	\captionsetup{font=small,width=0.8\textwidth}	
	\begin{center}		\includegraphics[width=2.5in]{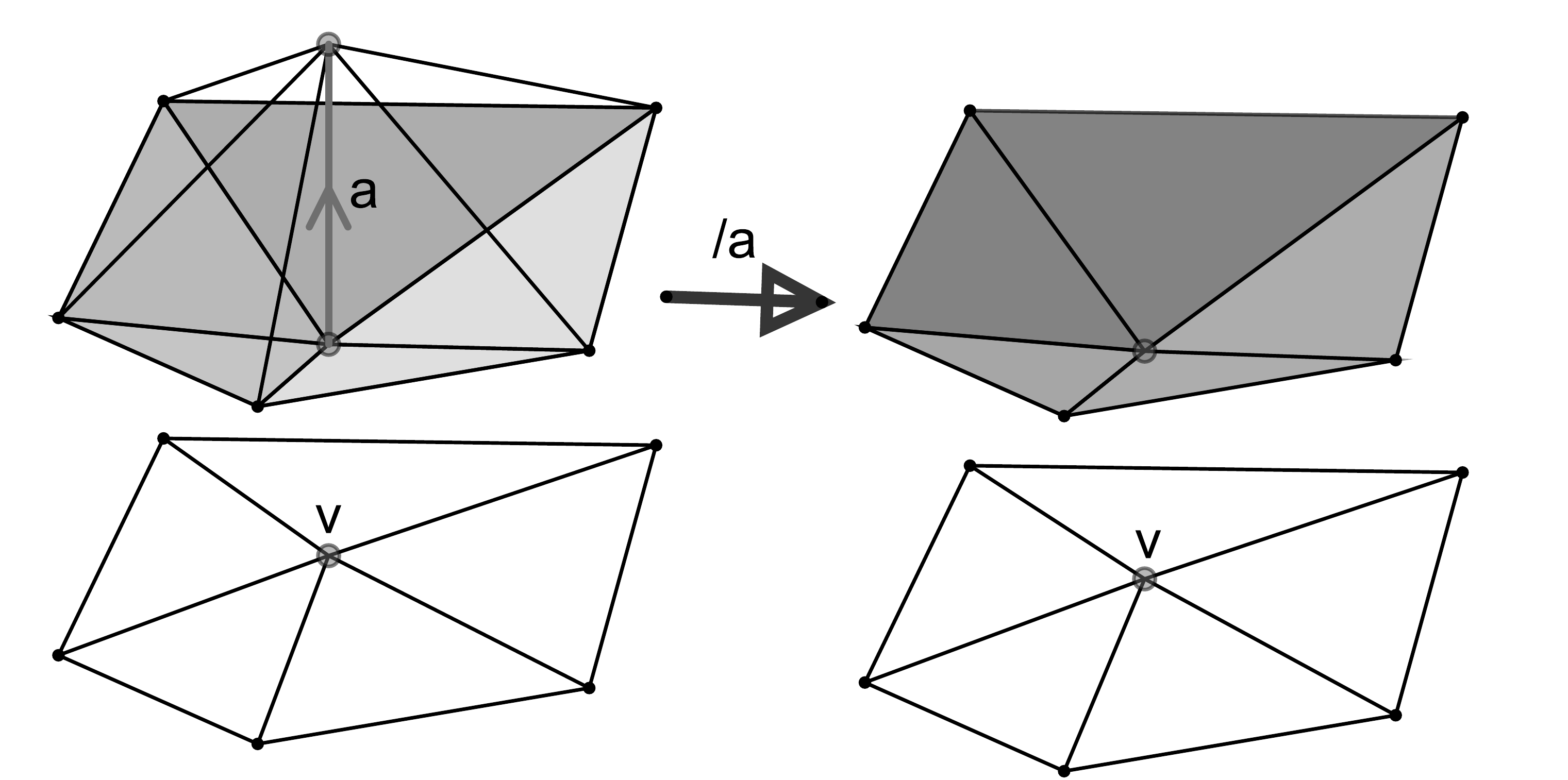}
		\caption{  \label{spind}}
	\end{center}
\end{figure}
	
\begin{figure}[h!]
	\captionsetup{font=small,width=0.8\textwidth}
	\begin{center}
		\includegraphics[width=4.0in]{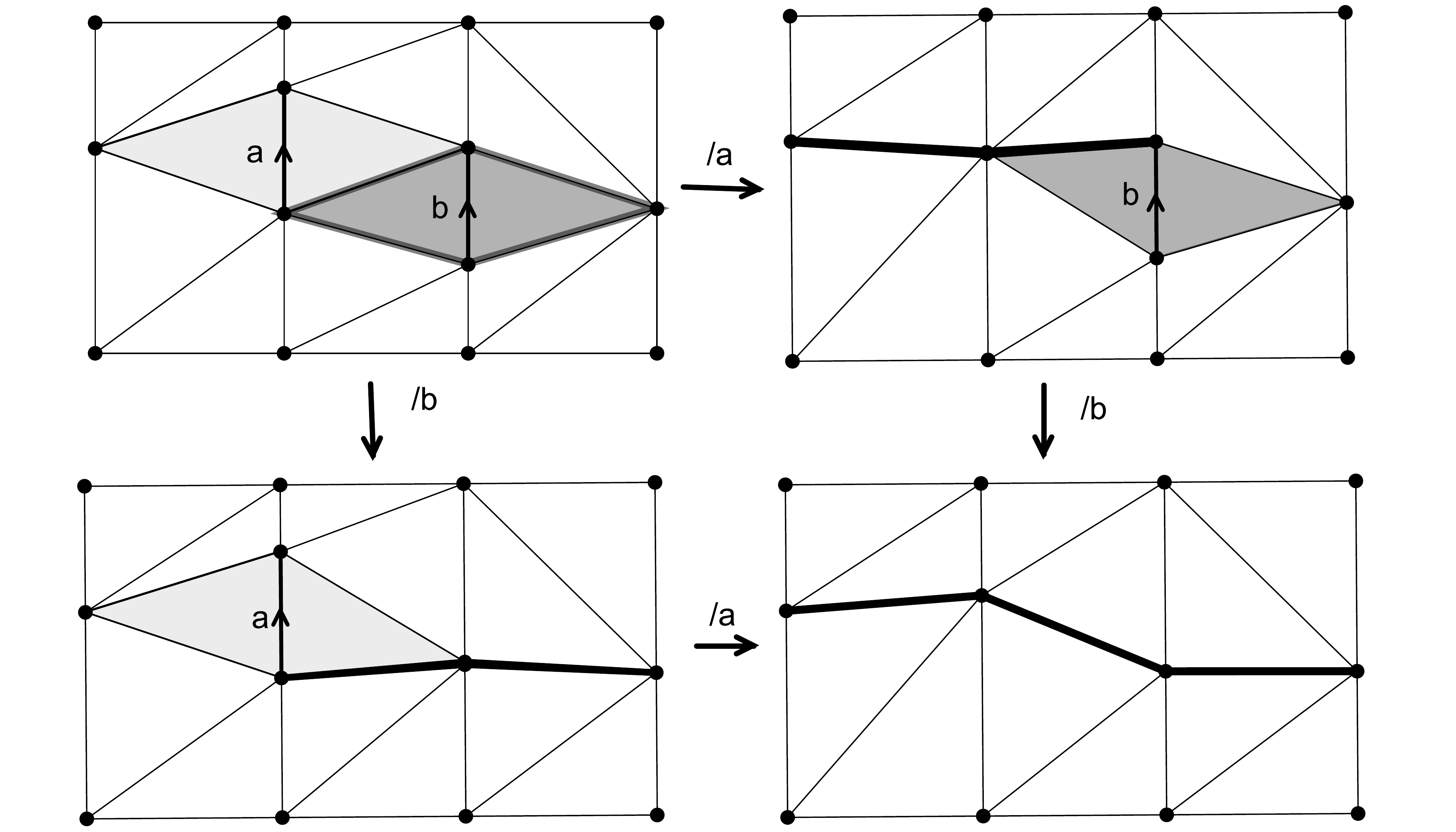}
		\caption{  \label{spind2}}
	\end{center}
\end{figure}

\begin{lemma}
If the circle $c(v)$ has more than one arc and $a\in c_1(v)$ - a chosen arc, then the spindle contraction $\pmb p/a$ is a correct s.c. bundle.
\end{lemma}
\begin{proof} In the simplicial language, the contraction can look  a bit puzzling,
but is obvious in the language of the local system of necklaces encoding the bundle $\pmb p$. One should take the bead corresponding to the arc  $a$ and remove this bead from all the necklaces having it in the boundary. The local system will remain correct if the bead was not the last in the necklace $\mc N(c(v))$ monochromatically colored  by $v$.
\end{proof}
\begin{proof}[Proof of Proposition \ref{conc}]
Take a s.c. bundle $\pmb E \times \la 1 \ra \xar{\pmb p \times \la 1\ra} \pmb B \times \la 1 \ra$. It has two copies  of bundle $\pmb p $ labeled by 0 and 1. Take the copy  1. Pick there the copy $a_1$ of arc $a$ over the copy $v_1$ of the vertex $v$ and perform the spindle contraction $\pmb p \times \la 1 \ra/a_1 $.  It coincides with $\pmb p/a$ on the 1-side and don't affect 0-side because  only the star is affected. One can do contractions  in any  order up to some circle of the bundle has more than two arcs. Now, select a single arc in every circle of the bundle and contract all the others by spindle contractions. The resulting strongly concordant minimal s.c. bundle will be completely determined by this selection.
\end{proof}

\section{Local  binary formula for Chern cocycle of minimal triangulations}\label{01}
\begin{itemize}
\item[{
} $\mb c_{01}$]Local  binary formula for Chern cocycle of minimal s.c. bunles is a universal simpicial Chern  2-cocycle on $\mb c_{01}\in Z_\triangle^2(\pmb SC;\{0,1\}\subset\Z)$ defined as the parity of a circular permutation of 3 elements
\begin{equation}\label{iy}
 \begin{gathered}
 \mb c_{01}\c{0,1,2} = \mb c_{01}(\pmb ec\c{0,1,2})=0 \\ \mb c_{01}\c{2,1,0}= \mb c_{01}(\pmb ec\c{2,1,0})=1
 \end{gathered}
\end{equation}
\end{itemize}
It is the rational local formula from  \cite{MS} shifted by the universal 2-coboundary $\frac{1}{2}$. Alternatively, it can be directly obtained from the exponential sequence of sheaves (\cite[Section 2.1]{Brylinski2008} ) using sections related to  the top or to the bottom hats of unique spindles
over the \v Cech (hyper)cover of the base by stars which has the initial semi-simplicial base complex as its \v Cech nerve. Alternatively, it can be guessed and checked as we know a geometrical triangulation of Hopf bundle \cite{Madahar:2000}.

\section{Huntington cyclic order axioms, Kan properties of $\pmb SC$ and binary Chern cocycle.}\label{hunti}
\p{Cyclic order.} Total cyclic order on a set is a way to inject the set into an oriented circle. We imagine the circle drawn on a paper to be oriented clockwise. For the finite set it is the way to complete  the set  up to graph cycle by introducing oriented edges between elements meaning  ``the next".  If the elements of a finite set  are numbered by the set $[k]$, then a total cyclic order on this set is  $\c{\omega} \in \pmb SC_k$ for some permutation of its elements  $\omega \in  \pmb S_k$.
Abstract total  cyclic order relation  was introduced by philosopher  E.V. Huntington \cite{Hunt1916, Hunt1935} as one of the fundamental orders of (Platonic) Universe. Another exposition is in  \cite{Novak1982}.  It can be axiomatically defined by ternary $HC(a,b,c)$ relaton. The  meaning of $HC(a,b,c)=$ ``True" is that the ordered triple $a,b,c$ sits on the circle clockwise. As Huntington put it: ``the arc running from $a$ to $c$ through $b$ in direction of arrow is less then complete circuit".
The independent axioms of a total cyclic order are:
\begin{itemize}
	\item[i] Cyclicity: If $HC(a, b, c)$ then $HC(b, c, a)$.
	\item[ii] Asymmetry: If $HC(a, b, c)$ then not $HC(c, b, a)$.
	\item[iii] Transitivity: If $HC(a, b, c)$ and $HC(a, c, d)$ then $HC(a, b, d)$.
	\item[iv] Totality: If $a, b$ and $c$ are distinct, then either $HC(a, b, c)$ or $HC(c, b, a)$.
\end{itemize}
For a  finite  set ordered by $[k]$  we can read the Huntington theory of total cyclic orders as follows.

The sets consisting of 1 and 2 ordered elements have a single total cyclic order. Let $k=2$ -- then axioms i,ii,iv are applicable and a Huntington total cyclic order on $[2]$ fixes one of two cyclic permutations of 3 elements: either even, i.e. $\c{0,1,2}$ (if $HC(0,1,2)$ holds) or odd, i.e $\c{2,1,0}$ (if $HC(2,1,0)$ holds).
Thus the two total Huntington cyclic orders on the ordered set $[2]$ fix and are fixed by the function $\mb c_{01}$ (\ref{iy}). It is the key observation.

  If $k=4$, then the transitivity axiom iii starts playing a role and gives a condition when 4 circular permutations of subtriples of $[3]$  fix the unique circular permutation of entire $[3]$.

  For all $k\geq 4$ the theory states that system of circular permutations of all subtriples which satisfy transitivity for all subquadruples, fix the unique circular order on entire $[k]$.

  We note that subsets of circular permutations are  simplicial
 boundaries in $\pmb SC$. Now   we can
translate the above observations into a form of Kan Extension Lifting Property for circular permutations over simplicial pairs $(\la k \ra,\partial \la k \ra ) $ in all the range with a single gap in dimension 3:
\begin{prop} \label{kan}
If $k = 0,1$ and $k \geq 4$ then any map of $k-1$-sphere $\partial \la k\ra \xar{\varphi} \pmb SC$ has a unique lift to a map  of $\la k \ra$, i.e. there exist a unique map $\la k \ra\xar{ \tilde \varphi} \pmb SC$ such that $\tilde \varphi|_{\partial \la  k \ra } = \varphi$. If $k=2$ then there exist two  different lifts. Dimension $k=3$ is exceptional.     	
\end{prop}

We can see that  $\pmb SC \approx K(\Z,2)$  by simplicial homotopy argument and we observe that it follows directly from its Huntington local axiomatic description above. The Proposition \ref{kan} states that the simplicial set $\pmb SC$ is minimal Kan contractible in all  dimensions except 2. Therefore it has homotopy groups $\pi_i$ vanishing  if $ i=0,1,\geq 3$. We need to know $\pi_2$. But  by Hurewicz theorem it amounts  computing second homology of $\pmb SC$ which is just the homology of the 2 sphere, by inspection of the hexagram on Figure \ref{hex}.

\p{Chern binary cocycle and cyclic order transitivity axiom.}
In Proposition \ref{kan} Huntington axioms became translated into a simplicial homotopy of $\pmb SC$ with a gap in dimension 3 where the transitivity axiom iii is not formulated topologically.   We can fill the gap a bit miraculously .

\bigskip
\noindent
Let
$f\in C_\triangle^2 (\partial \la 3 \ra; \{0,1\} \subset \Z)$ be an 
 0,1 valued integer cochain on the boundary of the ordered 3-simplex $\partial \la 3 \ra$. We can translate it into a Huntington cyclic relation $f^H$ on ordered subtriples of the set of vertices  $[3]$, which is the same as  fixing either even or odd cyclic permutations of the set of vertices of faces, which is the same as a singular map $\partial \la 3 \ra \xar{\ol f} \pmb SC $, which is the same as a minimal s.c. bundle $\ol f^* $ on $\partial \la 3 \ra$.

 We know that, if for the relation $f^H$ transitivity holds, then the cyclic orders on the triples assemble into a cyclic order on all [3], or equivalently $\ol f$ has extension over $\la 3 \ra$ and the minimal circle bundle $\ol f^*$  has extension to the minimal bundle $\tilde f^*$ over $\la 3 \ra$. For the bundle $\tilde f^*$, the cochain $f$ is its Chern binary cocycle  $\mb c_{01}( \tilde f^*) $ (Section \ref{01}) and therefore transitivity of $f^H$ implies that
\begin{equation}\label{coc}
\sum_{i=0}^4 (-1)^i f_i =0
\end{equation}
The inverse statement is true:
\begin{prop} \label{01har}
If the equation (\ref{coc}) holds, then
Huntington order  $f^H$ is transitive, or equivalently   the cyclic orders on triples are uniquely extendable  to a cyclic order on  $[3]$, or equivalently  the corresponding minimal bundle is uniquely extendable.
\end{prop}
\begin{proof}
	The proof is experimental. It is a pseudoscientific  check of cases  during meditation over the hexagram in Figure \ref{hex} representing the 3-skeleton $\pmb  SC(3)$. But the check is very short.
	 We  list in Table \ref{bin}  all $2^4=16$ binary 2-cochains $f$ on $\partial \la 3 \ra$ (in the order of 4-positional binary numbers)
	They correspond to 16 minimal s.c. bundles on $\partial \la 3 \ra$. The value
	$$\mb c(f)=\sum_{i=0}^4 (-1)^i f_i$$
	is the Chern number of the bundle $\tilde f^*$. It can be equal
	to $0$ (trivial bundle) $\pm 1$ (Hopf bundle with opposite orientations) $\pm 2$ (tangent bundle to the sphere $S^2$ with opposite orientations). Among them 6 are cocycles, i.e. minimal trivial bundles on $\partial \la 3 \ra$.
	In parallel we list all  6 minimal elementary bundles on the entire $\la 3 \ra$ corresponding to the circular permutations of 4 elements. Computing their boundary bundles gives exactly all the 6  minimal trivial bundles on $\partial \Delta^3$, corresponding to the 6 binary cocycles. This provides a 1-1 correspondence between 6 binary cocycles of $Z^2_{\triangle}(\partial \la 3\ra; \{0,1\} \subset \Z )$  and trivial s.c. bundles on $\partial \la 3 \ra$ extending to a minimal s.c. bundle on $\la 3 \ra$. The correspondence is presented in the table \ref{bin}:
\begin{table}[h]
	\scriptsize
	\captionsetup{font=small,width=0.8\textwidth}
	\centering
	\begin{tabular}{|c|c|c|c|| c|}
		\hline
		$f_0(123)$  & $f_1(023)$ &$f_2(013)$ & $f_3(012)$ &  \\ \hline
		+ & - & +  & - & $\mb c(f)$ \\ \hline \hline
		$\c{123}$  & $\c{023}$   & $\c{013}$  & $\c{012}$  & $\c{0123}$ \\             		
		0 & 0 & 0 & 0&\ul 0 \\ \hline \hline
		0 & 0 & 0 & 1 & -1\\ \hline
		0 & 0 & 1 & 0 & 1\\ \hline \hline
		$\c(231)$  & $\c{031}$   & $\c{031}$  & $\c{021}$  & $\c{0231}$ \\
		0 & 0 & 1 & 1 & \ul 0 \\ \hline \hline
		0 & 1 & 0 & 0 & -1 \\ \hline
		0 & 1 & 0 & 1 & -2 \\ \hline \hline
		$\c{312}$  & $\c{032}$   & $\c{031}$  & $\c{012}$  & $\c{0312}$ \\
		0 & 1 & 1 & 0 & \ul 0 \\ \hline \hline
		0 & 1 & 1 & 1 & -1 \\ \hline
		1 & 0 & 0 & 0 & 1 \\ \hline \hline
		$\c{213}$  & $\c{023}$   & $\c{013}$  & $\c{021}$  & $\c{0213}$ \\
		1 & 0 & 0 & 1 & \ul 0\\ \hline \hline
		1 & 0 & 1 & 0 & 2\\ \hline
		1 & 0 & 1 & 1 & 1 \\ \hline \hline
		$\c{132}$  & $\c{032}$   & $\c{013}$  & $\c{012}$  & $\c{0132}$ \\
		1 & 1 & 0 & 0 & \ul 0 \\ \hline \hline
		1 & 1 & 0 & 1 & -1 \\ \hline
		1 & 1 & 1 & 0 &  1 \\ \hline \hline
		$\c{321}$  & $\c{032}$   & $\c{031}$  & $\c{021}$  & $\c{0321}$ \\
		1 & 1 & 1 & 1 & \ul 0 \\ \hline
	\end{tabular}	
		\caption{  \label{bin}}
\end{table} 	
\end{proof}	
\section{Proof of Theorem 1.}\label{prt1}
Summarizing the achievements we
got:
\begin{prop} \label{hunt}

The set of minimal s.c. bundles on the base semi-simplicial set $\pmb B$ is in canonical one-to-one correspondence with local systems of circular permutations of ordered vertices of base simplices, simplicial maps $\on{Hom} (\pmb B, \pmb SC)$ and simplicial 2-cocycles in $Z^2_\triangle (\mf B; \{0,1\}\subset \Z)$.
\end{prop}
\begin{proof}
The first statements were discussed in p. \ref{neck}.\ref{cm}
The last statement we know from general Huntington theory
and Proposition \ref{01har}. The bundle defines uniquely the cocycle.

We need the inverse: a binary 2-cocycle uniquely defines bundle. Here is why it is true:
a binary 2-cocycle defines uniquely the local system of circular permutations on the 2 skeleton; by Proposition \ref{01har} the cocycle  condition provides transitivity of the system of cyclic orders, therefore it is uniquely extendable over the 3 skeleton. Now, by Kan property of cyclic orders from Proposition \ref{kan} the local system of circular permutations on the 3-skeleton is uniquely extendable  on the entire $\pmb B$.
\end{proof}

\begin{proof}[Proof of Theorem 1]
By Proposition \ref{conc} we know that any semi-simplicial
s.c. bundle triangulated over $\pmb B$ is concordant to a minimal one and therefore its Chern class is representable by a simplicial 2-cochain in the base with binary values.  By Proposition \ref{hunt} the inverse statement is true.
For classically simplicial triangulations the condition is necessary but not sufficient (see \cite{Mnev:2018}).
	
\end{proof}

\section{Effortless \& Local assembly of triangulated circle bundles with a prescribed Chern number over a closed oriented  triangulated surface.}\label{prt2}

For minimally triangulated circle bundles over triangulated oriented closed surfaces, we got a purely unobstructed free way of constructing triangulated circle  bundles with a prescribed Chern number.
\p{}
Suppose we have an oriented surface $M$, triangulated by a semi-simplicial complex $\pmb T$, $|\pmb T|=M$, $[M] \in Z^\triangle_2(\pmb T;\Z)$ -- fundamental class of $M$, fixing the orientation.
Then any 2-simplex  $x \in \pmb T_2$ obtains  relative positive or negative  orientation $o(x)\in \Z/2\Z$ according to value $(-1)^{o(x)}$ of fundamental class $[M]$ on the simplex $x$.
\begin{lemma}\label{ppar}
	Semi-simplicial triangulation of an oriented closed surface always has an even number of 2-simplices, half of them positively oriented, half of them - negatively.    	
\end{lemma}
\begin{proof}
	Pick a 1-cochain  $\mb 1_1 \in C_\triangle^1(\pmb T;\Z)$ having value 1 on every 1-simplex. Then  $d(\mb 1_1) = \mb 1_2$ is a coboundary in $B^2_\triangle (\pmb T; \Z)$ having value 1 on any 2-simplex. By Stokes' theorem, the  pairing is $\la \mb 1_2, [M]\ra =0$. Therefore, the number of positively oriented simplices is equal to the number of negatively oriented simplices and the total number $ \# \pmb T_2$ is even.
\end{proof}
\p{Proof of Theorem 2.}

\begin{proof}[Proof of Theorem \ref{thm02}]
	Take a simplicial 2-cochain $u\in C^2_\triangle(\pmb T;\{0,1\}\subset \Z)$. It will be a 2-cocycle since every 2-cochain in $C^2_\triangle(\pmb T;\Z)$ is a cocycle.    It defines an integer  $c(u)$, the element the second cohomology  group, by pairing with the fundamental cocycle: $$H^2(M;\Z)=\Z \ni c(u)= \la u, [M]\ra=\Sigma_{x\in \pmb T_2} (-1)^{o(x)}u(x) $$ By Lemma \ref{ppar}   $c(u)$ can be any integer number from the interval
	$$[-\frac{1}{2}\# \pmb T_2,...,-1,0,1,...,\frac{1}{2}\#\pmb T_2]$$
	Maximum value $\frac{1}{2}\# \mf B(2)$ of $c(u)$ is obtained by distributing 1's on the positively oriented simplices and 0's on the negatively oriented ones. Minimum value $-\frac{1}{2}\# \mf B(2)$ by distributing 1's on negatively oriented simplices, 0's on positively oriented ones.
	Picking the $u$ we can put the circular permutation $\c{0,1,2}$ on every 2-simplex $x$ having $u(x) =0$   and  $\c{2,0,1}$ on every simplex $x$ having $u(x)=1$. We effortlessly got a local system of necklaces since  boundaries  of circular permutations of  $3$ elements  are always the trivial circular permutations $\c{0,1},\c{0}$ and always fit because they are always of single type and have no automorphisms. Therefore, replacing circular permutations by elementary s.c. bundles, we obtain a minimal s.c. bundle having $u$ as its Chern cocycle and $c(u)$ as its Chern number. According to Proposition \ref{conc}, any bundle triangulated over  $\pmb T$ is concordant to a minimal one and therefore   has a binary simplicial representative of its Chern cocycle.

\bigskip
\noindent	
	By the argument from \cite[Lemma 0.1]{Mnev:2018} the bundle with maximal possible Chern number  $\frac{1}{2}\# \mf T_2$ can be only semi-simplically triangulated over the $\mf T$.  \end{proof}
It would be interesting to investigate more the case of classically simplicial triangulations. In view of Proposition \ref{puresimp} it seems that the spindle contraction  reduces a classically simplicial bundle triangulation to a classically simplicial triangulation with  only several possible types of elementary subbundles.  So the analysis  can be accessible.

\def\cprime{$'$} \def\cprime{$'$}

\end{document}